\documentclass[12pt,leqno]{amsart}
\usepackage{amsmath,amsfonts,amssymb,cite}
\usepackage[left=1 in,top=1 in,right=1 in,bottom=1 in]{geometry}

\renewcommand{\div}{\mbox{div}\,}
\renewcommand{\div}{\mbox{div}\,}

\newcommand{\R}{{\mathbb R}} \newcommand{\K}{{\mathbb K}}

\newcommand{\A}{\mathbf{A}}

\newcommand{\ba}{\mathbf{a}}

\newcommand\norm[1]{\left\| #1\right\|}

\def\longequals{\mathbin{=\kern-2pt=}}
\def\eqdef{\mathbin{\buildrel \rm def \over \longequals}}

\newtheorem{theorem}{Theorem}[section]

\newtheorem{remark}[theorem]{Remark}
\newtheorem{lemma}[theorem]{Lemma}
\newtheorem{proposition}[theorem]{Proposition}

\numberwithin{equation}{section}

\newcommand{\beq}{\begin{equation}}
\newcommand{\eeq}{\end{equation}}
\newcommand{\beqs}{\begin{equation*}}
\newcommand{\eeqs}{\end{equation*}}

\def\longequals{\mathbin{=\kern-2pt=}}
\def\eqdef{\stackrel{\rm def}{=}}

\begin{document}
\title[Local  gradient estimates]{Local  gradient estimates   for degenerate elliptic equations}

\author[L. Hoang]{Luan Hoang$^{\dag}$}
\address{$^\dag$ Department of Mathematics and Statistics, Texas Tech University, Box 41042, Lubbock, TX 79409--1042, U.S.A.}
\email{luan.hoang@ttu.edu}

\author[T. Nguyen]{Truyen  Nguyen$^\ddag$} 
\address{$^\ddag$Department of Mathematics, University of Akron, 302 Buchtel Common, Akron, OH 44325--4002, U.S.A}
\email{tnguyen@uakron.edu}

\author[T. Phan]{Tuoc  Phan$^{\dag\dag}$}
\address{$^{\dag\dag}$Department of Mathematics, University of Tennessee, Knoxville, 227 Ayress Hall, 1403 Circle Drive, Knoxville, TN 37996, U.S.A. }
\email{phan@math.utk.edu}

\date{\today}

\begin{abstract} 
This paper is focused on the local interior $W^{1,\infty}$-regularity for  weak solutions of  degenerate elliptic equations of the form $\text{div}[\ba(x,u, \nabla u)] +b(x, u, \nabla u) =0$, which 
include those of $p$-Laplacian type. We derive an explicit estimate of the local \textup{$L^\infty$}-norm for the solution's gradient in terms of its local $L^p$-norm.  
Specifically, we prove
\beqs
\norm{\nabla u}_{L^\infty(B_{\frac{R}{2}}(x_0))}^p \leq \frac{C}{|B_R(x_0)|}\int_{B_R(x_0)}|\nabla u(x)|^p dx.
\eeqs
This estimate   paves the way for our forthcoming work \cite{HNP2}  in  establishing   $W^{1,q}$-estimates (for $q>p$) for weak solutions to  a much larger class of quasilinear elliptic equations.

\end{abstract}

\maketitle

\section{Introduction}
Consider the Euclidean space $\R^n$ with integer $n\geq 1$.
Denote $B_R(x) =\{y\in \R^n:\, |y-x|<R\}$ and $B_R = B_R(0)$.
In this paper we investigate local gradient estimates 
for weak solutions to equations of divergence form
\begin{equation} \label{ref-eqn}
\text{div}[\ba(x,u, \nabla u)] +b(x, u, \nabla u) =0 \quad \text{in} \quad B_3,
\end{equation}
where the vector field $\ba$ and the function $b$ satisfy certain ellipticity and growth conditions. Specifically, let $\K \subset \mathbb{R}$ be an interval, 
and let $\ba =(\ba^1,\dots, \ba^n) : B_3 \times \K \times \mathbb{R}^n \rightarrow 
\mathbb{R}^n$  and 
 $b: B_3 \times \K \times \mathbb{R}^n \rightarrow 
\mathbb{R}$ be  Carath\'eodory maps such that $\ba$ is  differentiable on $B_3\times \K \times (\mathbb{R}^n \setminus\{0\})$.
We assume also that
\beq\tag{H1}
\ba (x,z,0)=0 \qquad\qquad \qquad\qquad \qquad\qquad \qquad\qquad 
\forall \ (x, z) \in B_3 \times \K;
\eeq
and there exist $p>1$ and $\gamma_0,\gamma_1>0$ such that
\beq\tag{H2}
\sum_{i, k =1}^n \frac{\partial \ba^i(x,z,\eta)}{\partial \eta_k} \xi_i \xi_k \geq \gamma_0|\eta|^{p-2}|\xi|^2 \qquad\quad 
\forall \ (x, z,\eta,\xi) \in B_3 \times \K\times (\mathbb{R}^n\setminus \{0\})\times \R^n;
\eeq
\beq\tag{H3}
\sum_{k=1}^n \big| \frac{\partial \ba(x, z, \eta)}{\partial \eta_k} \big| \leq \gamma_1 |\eta|^{p-2} \qquad\qquad\qquad\qquad\quad \forall \ (x, z,\eta) \in B_3\times \K\times( \mathbb{R}^n\setminus \{0\});
\eeq
\beq\tag{H4} 
\sum_{i=1}^n\big|\frac{\partial \ba}{\partial x_i}(x,z,\eta)\big |  +
|\eta| \big |\frac{\partial \ba}{\partial z}(x,z,\eta)\big | \leq \gamma_1(|\eta|^{p-1}+|\eta|^p) \quad\forall  (x, z,\eta) \in B_3 \times \K\times\mathbb{R}^n;
\eeq
\beq\tag{H5} 
 |b (x,z,\eta)|  \leq \gamma_1(|\eta|^{p-1}+|\eta|^p) \qquad\qquad\qquad\qquad \qquad\qquad \forall  (x, z,\eta) \in B_3 \times \K\times \mathbb{R}^n.
\eeq
We would like to stress that \textup{(H1)}--\textup{(H5)} are  only assumed to hold for $z\in \K$ which might be a strict subset of $\R$, and the constants $\gamma_0, \gamma_1$ can  depend on $\K$. 
For example, in some cross-diffusion equations in population dynamics  (see \cite{HNP1} and the references therein), we have $p=2$, $\ba(x,z,\eta) = (1 +z) \eta$, and  $\K$ is  a bounded subset of $(0, \infty)$.   

A weak solution $u(x)$ of \eqref{ref-eqn} is defined to be a function in $W^{1,p}_{\text{loc}}(B_3)$  that satisfies $u(x)\in \K $ for a.e. $x\in B_3$, 
and
\beqs
-\int_{B_3} \ba (x,u,\nabla u) \cdot \nabla \varphi(x) \, dx + \int_{B_3} b(x,u,\nabla u) \varphi(x) \, dx= 0\quad \forall \varphi \in W_0^{1,p}(B_3)\cap L^\infty(B_3).
 \eeqs

The equations of the form \eqref{ref-eqn}  have been studied extensively in the literature, see  
\cite{D, E1, La, Le, M, Serrin, Tol1, Tol2, Uh, Ur}. In particular, interior $C^{1,\alpha}$ regularity for homogeneous $p$-Laplace  equations  was established by Uraltceva \cite{Ur}, Uhlenbeck \cite{Uh}, Evans \cite{E1} and Lewis \cite{Le}.  Regarding the local regularity for  general quasilinear equations \eqref{ref-eqn}, the following  classical result is proved by DiBenedetto \cite{D} and Tolksdorf\cite{Tol2}.
\begin{theorem}\label{2p-Linf}{\rm(\cite[Theorem~1]{D},\cite[Theorem~1]{Tol2})}   Assume 
\textup{(H1)--(H3)}, and
\beq\tag{H4$'$} 
\sum_{i=1}^n \big |\frac{\partial \ba}{\partial x_i}(x,z,\eta)\big |  +
\big |\frac{\partial \ba}{\partial z}(x,z,\eta)\big | \leq \gamma_1|\eta|^{p-1},
\eeq
\beq\tag{H5$'$} 
 \left |b(x,z,\eta)\right | \leq \gamma_1 |\eta|^p 
 \eeq
 hold for every $ (x, z,\eta) \in B_3 \times \K\times \mathbb{R}^n$. If $u$ is a bounded weak solution of \eqref{ref-eqn}, then  $u\in C^{1,\alpha}_{loc}(B_3)$ and there exists a constant $M>0$ depending only on $n, p, \gamma_0,\gamma_1$ and $\|u\|_{L^\infty (B_3)}$ such that 
\begin{equation}\label{rough}
\norm{\nabla u}_{L^\infty(B_2)} \leq M.
\end{equation}
\end{theorem}

Our purpose is to explicate estimate \eqref{rough}, namely, to bound the local $L^\infty$-norm of $|\nabla u|$ by its local $L^p$-norm that preserves the scaling in $x$. Our achieved result holds for  more general vector field $\ba(x, u,\nabla u)$ and function  $b(x, u,\nabla u)$  than the ones required in Theorem~\ref{2p-Linf}.
Precisely, we obtain: 
\begin{theorem}\label{Lip} Assume that \textup{(H1)--(H5)} hold.
Let $u$ be a  weak solution of \eqref{ref-eqn} that satisfies 
\beq\label{rough-ass}
\|u\|_{L^\infty(B_{\frac{11}{4}})} \leq M_0.
\eeq
Then there exists $C>0$ depending only on $n$, $p$,  $\gamma_0$, $\gamma_1$  and $M_0$ such that
\beq\label{Win}
\norm{\nabla u}_{L^\infty(B_{\frac{R}{2}}(x_0))}^p \leq \frac{C}{|B_R(x_0)|}\int_{B_R(x_0)}|\nabla u(x)|^p dx, \quad 
\forall x_0\in B_1,\  0 < R \leq 1.
\eeq
\end{theorem}

When the growths of $\ba$ and $b$ in the $\eta$ variable are weaker, the assumption   
\eqref{rough-ass} on the local boundedness of the solution can be dropped. In particular, we obtain the following result when  conditions (H4) and (H5) are strengthened appropriately.

\begin{theorem} \label{Lip-plus} Assume \textup{(H2)--(H3)}, and
\begin{equation}\label{strengthened-cond}
\sum_{i=1}^n\big|\frac{\partial \ba}{\partial x_i}(x,z,\eta)\big|  +
|\eta| \big |\frac{\partial \ba}{\partial z}(x,z,\eta)\big |+|b (x,z,\eta)| \leq \gamma_1|\eta|^{p-1}, \quad \forall  (x, z,\eta) \in B_3 \times \K\times \mathbb{R}^n.
\end{equation}
Then there exists $C=C(n,p,\gamma_0,\gamma_1)>0$ such that for any weak solution  $u$  of \eqref{ref-eqn}, the estimate \eqref{Win} holds true.
\end{theorem}

Gradient estimates of the type 
 \eqref{Win}  were discovered by Uhlenbeck \cite{Uh} for    elliptic systems of the form $\div \big(\A(|\nabla u|^2)\cdot \nabla u \big)=0$, and were later extended further by Tolksdorf \cite{Tol1} for a larger class of quasilinear elliptic systems. In 
 \cite[Proposition~3.3]{D}, DiBenedetto derived estimate  \eqref{Win} for weak solutions to scalar equation $\div \ba(\nabla u)=0$. The same estimate was established in \cite[Lemma~1.1]{BDM} for equations of the form  $\div \big(|\nabla u|^{p-2} \nabla u\big)  + b(x,u,\nabla u)=0$ with $p>1$ and $b$ satisfying the growth condition $|b(x,z,\eta)|\leq \gamma_1 |\eta|^{p-1}$. Thus, our Theorem~\ref{Lip-plus} generalizes the result obtained in  \cite{D,BDM}. The significance  of our main result 
 in  Theorem~\ref{Lip} is that it holds true for  the general equation \eqref{ref-eqn} with $\ba$, $b$ depending on $x$, $z$ and having general structure (H1)--(H5).

Our main motivation for deriving the local gradient estimates   in Theorems~\ref{Lip} and \ref{Lip-plus} is to be able to establish   $W^{1,q}$-estimates (for $q>p$) for weak solutions to  a large class of equations of the form $\div \A(x,u,\nabla u) + B(x,u,\nabla u) = \div \mathbf F$, where the vector field $\A$ is 
allowed to be discontinuous in $x$, Lipschitz continuous in $u$ and its  growth in the gradient variable  is  like the $p$-Laplace operator with $1<p<\infty$. This is achieved 
in our forthcoming work \cite{HNP2} by using Caffarelli-Peral perturbation technique \cite{CP}, and  the
 quantified estimate \eqref{Win} for \eqref{ref-eqn} plays an essential role in performing that process.

The proofs of Theorems \ref{Lip} and \ref{Lip-plus} will be given in section \ref{sec:proof-main-theorem}, after some preparations in sections \ref{pre} and \ref{inter}. We prove them by employing  standard iteration and interpolation techniques together with  refining  some results presented in \cite{D, La}.
However, some lower order terms arising from the $x$, $z$ dependence are treated carefully and  differently (see \eqref{b-i} below) compared to the known work in order to obtain the desired homogeneous estimate.

\section{Preliminary estimates} \label{pre}

In this section we always assume that $u$ is a weak solution of \eqref{ref-eqn}.
 We  begin with a result which is a  simple 
modification of \cite[pages 834-835]{D}. Throughout the paper, 
we  denote $w = |\nabla u|^2$ and $|\nabla^2 u| =(\sum_{i,j=1}^n |u_{x_ix_j}|^2)^{1/2}$.

\begin{lemma} \label{test-lemma} Assume that \textup{(H2)--(H5)} hold. There exists a constant $C >0$ depending only on $n$, $\gamma_0$ and $\gamma_1$ such that
\begin{multline}\label{Iest}
\int_{B_3} w^{\frac{p-2}{2}} |\nabla^2 u|^2 \beta(w) \xi^2 dx 
+ \int_{B_3} w^{\frac{p-2}{2}}|\nabla w|^2\beta'(w)\xi^2 dx \\
\leq C  \left\{ \int_{B_3} \big(w^{\frac{p-2}{2}} |\nabla w| + w^{\frac{p}{2}} + w^{\frac{p+1}{2}}\big) |\nabla \xi| \beta(w) \xi dx +
\int_{B_3} (w^{\frac{p}{2}} + w^{\frac{p+2}{2}})\big[\beta(w) + w\beta'(w) \big] \xi^2 dx
 \right\}
\end{multline}
 for any nonnegative function $\xi \in C_0^\infty(B_3)$ and any $\beta\in \text{Lip}_{loc}([0,\infty))$ satisfying   
$\beta,\beta' \geq 0$.  
\end{lemma}
\begin{proof} Using the difference-quotient argument as indicated in \cite{Uh} (or 
\cite[Proposition 1]{Tol2}) or using the approximation procedure as in \cite{D}, we may assume that $u \in C^2(B_3)$ and $|\nabla u(x)|>0$ for every 
$x\in B_3$. For each $i =1,2,\dots, n$, define
\begin{equation} \label{b-i}
\mathbf{b}_i(x,z, \eta) =  \frac{\partial \ba}{\partial x_i}(x, z, \eta) + 
 \frac{\partial \ba}{ \partial z} (x,z, \eta) \eta_i, \quad 
\forall  \ (x,z,\eta) \in B_3\times \K\times \mathbb{R}^n. 
\end{equation}
 By differentiating equation \eqref{ref-eqn} with respect to $x_i$, we have
\[
\text{div}\Big [ \sum_{j=1}^n u_{x_ix_j} \frac{\partial \ba}{\partial \eta_j}(x, u, \nabla u) +
\mathbf{b}_i(x,u,\nabla u) \Big] + \frac{d}{d x_i} b(x,u, \nabla u)=0 \quad \text{in} \quad B_3
\]
in the weak sense.
Using $\varphi = u_{x_i}\beta(w) \xi^2$ as a test function in the weak formulation  and 
summing over $i =1, 2,\dots, n$, we  obtain
\begin{multline} \label{test-1}
 \sum_{i,j,k =1}^n\int_{B_3}\frac{\partial \ba^k}{\partial \eta_j}u_{x_ix_j}\Big[u_{x_ix_k} \beta(w)\xi^2 +u_{x_i} w_{x_k} \beta'(w) \xi^2 + 2u_{x_i} \beta(w) \xi \xi_{x_k} \Big] dx \\
 = -\sum_{i=1}^n \int_{B_3} \Big[\mathbf{b}_i(x,u,\nabla u)\cdot \nabla \varphi + b(x,u,\nabla u)  \varphi_{x_i}\Big]dx. 
\end{multline}

Dealing with the LHS of \eqref{test-1}, we have  from assumptions (H2) and (H3) that
\beqs  \sum_{i,j,k =1}^n\int_{B_3}\frac{\partial \ba^k(x,u,\nabla u)}{\partial \eta_j} u_{x_ix_j} u_{x_ix_k} \beta(w)\xi^2dx  \geq \gamma_0
\sum_{i=1}^n \int_{B_3} w^{\frac{p-2}{2}} |\nabla u_{x_i}|^2 \beta(w) \xi^2 dx,
\eeqs
\beqs \sum_{i,j,k =1}^n\int_{B_3}\frac{\partial \ba^k(x,u,\nabla u)}{\partial \eta_j}u_{x_ix_j} u_{x_i} w_{x_k} \beta'(w) \xi^2 dx  
\geq \frac{\gamma_0}{2} \int_{B_3} w^{\frac{p-2}{2}} |\nabla w|^2 \beta'(w) \xi^2 dx,
\eeqs
and
\beqs 
\Big | \sum_{i,j,k =1}^n \int_{B_3}\frac{\partial \ba^k(x,u,\nabla u)}{\partial \eta_j}u_{x_ix_j} u_{x_i} \beta(w) \xi \xi_{x_k} dx\Big| 
 \le \frac{n\gamma_1}{2} \int_{B_3} w^{\frac{p-2}{2}} |\nabla w| |\nabla \xi| \beta(w) \xi  dx.
\eeqs
Therefore, 
\begin{multline} \label{RHS}
\text{LHS of \eqref{test-1}} 
\geq \gamma_0 \int_{B_3} w^{\frac{p-2}{2}} |\nabla^2 u|^2 \beta(w) \xi^2 dx  + 
\frac{\gamma_0}{2} \int_{B_3} w^{\frac{p-2}{2}} |\nabla w|^2 \beta'(w) \xi^2 dx  \\
 -n\gamma_1 \int_{B_3} w^{\frac{p-2}{2}} |\nabla w| |\nabla \xi| \beta(w) \xi  dx. 
\end{multline}

For the RHS of \eqref{test-1}, note that
\beqs
\sum_{i=1}^n |\nabla \varphi_i|\le  C_n \Big(|\nabla^2 u| \beta(w) \xi^2 +|\nabla u| \beta'(w) |\nabla w|\xi^2+|\nabla u||\nabla \xi| \beta(w) \xi\Big),
\eeqs
and from  (H4)--(H5) that
\begin{equation}\label{eq:b_i-b}
 |\mathbf b_i(x,u,\nabla u)|+|b(x,u,\nabla u)|\le 2\gamma_1 (w^{\frac{p-1}{2}} + w^{\frac{p}{2}}).
\end{equation}
Therefore, there exists a constant $C=C(n,\gamma_1)>0$ such that
\beqs 
\text{RHS of \eqref{test-1}} 
\le C   \int_{B_3} (w^{\frac{p-1}{2}} + w^{\frac{p}{2}})\Big(|\nabla^2 u| \beta(w) \xi^2 +w^{\frac12} \beta'(w) 
|\nabla w|\xi^2+w^{\frac12}|\nabla \xi| \beta(w) \xi\Big)dx.
\eeqs
We then estimate for $\epsilon>0$ that 
\begin{align*}
C   \int_{B_3} (w^{\frac{p-1}{2}} + w^{\frac{p}{2}}) |\nabla^2 u| \beta(w) \xi^2 dx
& \le  \epsilon  \int_{B_3} w^{\frac{p-2}{2}} |\nabla^2 u|^2 \beta(w)\xi^2 dx  \\
&\quad + C_\epsilon\int_{B_3} (w^{\frac{p}{2}} + w^{\frac{p+2}{2}} )\beta(w) \xi^2 dx, 
\end{align*}
\begin{align*}
C  \int_{B_3} (w^{\frac{p}{2}} + w^{\frac{p+1}{2}}) \beta'(w) |\nabla w|\xi^2dx
&\le  \epsilon \int_{B_3} w^{\frac{p-2}{2}} |\nabla w|^2 \beta'(w) \xi^2 dx \\
&\quad + C_\epsilon  \int_{B_3} (w^{\frac{p+2}{2}} + w^{\frac{p+4}{2}}) \beta'(w) \xi^2 dx.
\end{align*}
Consequently,
\begin{multline} \label{LHS}
\text{RHS of \eqref{test-1}} 
\leq \epsilon   \int_{B_3} w^{\frac{p-2}{2}} |\nabla^2 u|^2 \beta(w)\xi^2 dx  
+ \epsilon \int_{B_3} w^{\frac{p-2}{2}} |\nabla w|^2 \beta'(w) \xi^2 dx \\
 + C_\epsilon\Big[ \int_{B_3} (w^{\frac{p}{2}} + w^{\frac{p+2}{2}}) [\beta(w) + w\beta'(w) ] \xi^2 dx
+ \int_{B_3} (w^{\frac{p}{2}} + w^{\frac{p+1}{2}}) |\nabla \xi| \beta(w) \xi dx \Big].
\end{multline}
The lemma then follows from 
 \eqref{RHS} and  \eqref{LHS}  by taking $\epsilon=\gamma_0/4$.
\end{proof}

As a consequence of Lemma~\ref{test-lemma}, we obtain: 
\begin{lemma} \label{DiGiorgi-Class} Assume that \textup{(H2)--(H5)} hold. Let $v = w^{p/2} = |\nabla u|^p$. Then there 
exists  $C=C( n, p, \gamma_0,\gamma_1)>0$ such that 
\[
 \int_{B_3}|\nabla (v -k)^+|^2 \xi^2 dx \leq C \int_{B_3} [(v-k)^+]^2|\nabla\xi|^2 dx
+C \int_{B_3} (w^p + w^{p+1})\chi_{v >k}(x)\xi^2 dx
\]
for every constant $k>0$ and  every nonnegative function $\xi \in C_0^\infty(B_3)$.
\end{lemma}
\begin{proof} We  apply Lemma~\ref{test-lemma} with $\beta(s) = (s^{p/2} - k)^+$. Then by dropping the first term in \eqref{Iest}  and using
$\beta(w)+w\beta'(w) \le (1+\frac{p}{2})w^{p/2}\chi_{v>k}$,  we obtain
\begin{align*}
  &\frac{p}{2}\int_{B_3} w^{p-2}|\nabla w|^2\chi_{v >k} (x)\xi^2 dx\\ 
&\leq C \left\{   \int_{B_3}\big(w^{\frac{p-2}{2}}|\nabla w|+ w^{\frac{p}{2}} + w^{\frac{p+1}{2}} \big) \xi (v-k)^+ |\nabla \xi|  dx   +\frac{p+2}{2}\int_{B_3}(w^p +w^{p+1}) \chi_{v >k}(x) \xi^2 dx
 \right\}.
\end{align*}
The lemma then follows from Cauchy-Schwarz's inequality and  the fact that
\[
\int_{B_3} w^{p-2}|\nabla w|^2\chi_{v >k} (x)\xi^2 dx 
= \frac{4}{p^2}\int_{B_3} |\nabla (v-k)^+|^2\xi^2 dx.
\]
\end{proof}
\begin{remark}\label{rm:better-est}
If we assume  \eqref{strengthened-cond} in place of {\rm (H4)--(H5)}, then \eqref{eq:b_i-b} becomes
 \[|\mathbf b_i(x,u,\nabla u)|+|b(x,u,\nabla u)|\le \gamma_1 w^{\frac{p-1}{2}}.
  \]
Then by inspecting the proof we see that  \eqref{Iest} holds without the terms $w^{\frac{p+1}{2}}$ and $w^{\frac{p+2}{2}}$.
As a consequence, instead of 
Lemma~\ref{DiGiorgi-Class} we now obtain
\begin{equation}\label{strengthened-est}
 \int_{B_3}|\nabla (v -k)^+|^2 \xi^2 dx \leq C \int_{B_3} [(v-k)^+]^2|\nabla\xi|^2 dx
+C \int_{B_3} w^p \chi_{v >k}(x)\xi^2 dx
\end{equation}
for every constant $k>0$ and  every nonnegative function $\xi \in C_0^\infty(B_3)$.
\end{remark}

The next lemma  gives an  estimate for $\|\nabla u\|_{L^p}$ in terms of $\|u\|_{L^\infty}$.
\begin{lemma}  \label{Lp-est} Assume that \textup{(H1)--(H3)} and \textup{(H5)} hold. There exists  a constant $C>0$ depending only on $p$, $n$, $\gamma_0$ and $\gamma_1$ such that
\begin{equation}
\label{W^{1,p}-L^infty}
\int_{B_\frac{r}{2}(x_0)} |\nabla u|^p dx \leq C  r^n \big(r^{-p} +1\big) e^{C\|u\|_{L^\infty(B_{r}(x_0))}}\quad \mbox{for every}\quad  B_{r}(x_0) \subset B_3.
\end{equation}
\end{lemma}
\begin{proof} We follow the arguments in the proof of \cite[Lemma 1.1, p. 247]{La}.  Let $M=\|u\|_{L^\infty(B_{r}(x_0))}$. Since \eqref{W^{1,p}-L^infty} is trivial if $M=\infty$, we can  assume that $M<\infty$. 
Let $\xi \in C_0^\infty(B_{r}(x_0))$  be the standard cut-off function  with $\xi =1$ on 
$B_{\frac{r}{2}}(x_0)$ and $|\nabla \xi| \leq \frac{c}{r}$.  Then for any $\lambda >0$, by taking $e^{\lambda u} \xi^p$ as a test function we obtain
\[
\int_{B_{r}(x_0)} e^{\lambda u} \Big[ \lambda( \ba \cdot \nabla u) \xi^p + 
p( \ba \cdot \nabla \xi) \xi^{p-1} \Big ] dx =\int_{B_{r}(x_0)} b(x,u,\nabla u)\, e^{\lambda u}\xi^p dx.
\]
Note that as a consequence of (H1)--(H3), we have
\[
\ba(x, u, \nabla u) \cdot \nabla u \geq \gamma_0 |\nabla u|^p \quad \text{and} \quad 
|\ba(x, u, \nabla u) \cdot \nabla \xi| \xi^{p-1}\leq \gamma_1 |\nabla u|^{p-1} |\nabla \xi|\xi^{p-1}.
\]
These together with condition (H5)  give
\begin{align*}
\big(\lambda \gamma_0 -\gamma_1\big) \int_{B_{r}(x_0)}  e^{\lambda u}  |\nabla u|^{p} \xi^p dx 
&\leq \int_{B_{r}(x_0)}e^{\lambda u} \Big(  p \gamma_1 |\nabla u|^{p-1}\xi^{p-1} |\nabla \xi| +\gamma_1 |\nabla u|^{p-1}\xi^{p}\Big)dx \\
&\leq C\int_{B_{r}(x_0)} e^{\lambda u} \Big(|\nabla u|^p \xi^p + |\nabla \xi|^p  + \xi^p\Big) dx,
\end{align*}
where $C$ depends only 
on $p$ and $\gamma_1$.
Choosing $\lambda = (\gamma_1+2C)/\gamma_0$, we then get
\begin{equation*} \label{Lp-La}
\int_{B_{r}(x_0)}|\nabla u|^p \xi^p dx \leq e^{\frac{2(\gamma_1 + 2C) M}{\gamma_0}}\int_{B_{r}(x_0)} \big(|\nabla \xi|^p +\xi^p\big) dx
\leq   e^{\frac{2(\gamma_1 + 2C) M}{\gamma_0}} (c^p r^{-p} +1)|B_{r}(x_0)|.
\end{equation*}
 This yields \eqref{W^{1,p}-L^infty} as desired since $\xi =1$ on $B_{\frac{r}{2}}(x_0)$.
\end{proof}

We close the section by recalling a result about H\"older estimates for solutions to \eqref{ref-eqn}.
\begin{theorem}\label{thm:Holder} {\rm(\cite[Theorem~1.1, page 251]{La})}
Assume that \textup{(H1)--(H3)}  and  \textup{(H5)} hold. Let $u$ be a  weak solution of \eqref{ref-eqn} that satisfies \eqref{rough-ass}. Then there exist 
 constants $C_0>0$ and $\alpha\in (0,1)$ depending only on $n$, $p$, $\gamma_0$, $\gamma_1$ and $M_0$ such that
\[
|u(x) -u(y)|\leq C_0 |x-y|^\alpha \quad \mbox{for every}\quad x,y\in B_{\frac{21}{8}}.
\]
\end{theorem}

\section{Interpolation inequalities}\label{inter}

In this section we collect some known interpolation results which will be used later.
 We note that they are independent of the PDE under consideration.  
\begin{lemma}{\rm(\cite[Lemma~4.5, Chapter~2]{La}  and \cite[Lemma~2.4]{D})}
\label{embed-iq}
Let $p>1$, $\rho>0$, and $f\in C^2(\overline {B_\rho(x_0)})$ satisfy $|\nabla f|>0$.
Then for any $\xi\in C_0^1(B_\rho(x_0))$, we have
\[
\int_{B_\rho(x_0)}|\nabla f|^{p+2} \xi^2 dx \leq 2 (\sqrt{n} +p)^2 \big(\text{osc}_{B_\rho(x_0)} f\big)^2   \int_{B_\rho(x_0)}\Big[
|\nabla f|^{p-2}|\nabla^2 f|^2 \xi^2   
+ |\nabla f|^{p}  |\nabla \xi|^2\Big]dx ,
\]
where $\text{osc}_{B_\rho(x_0)} f =\sup_{x\in B_\rho(x_0)}{|f(x) - f(x_0)|}$.
\end{lemma}
\begin{proof}  We include a proof for the sake of completeness. Let $v = |\nabla f|^2$. Then
\[
\int_{B_\rho(x_0)} v^{\frac{p+2}{2}} \xi^2 dx = \int_{B_\rho(x_0)} v^{\frac{p}{2}}|\nabla f|^2 \xi^2 dx =  \int_{B_\rho(x_0)} v^{\frac{p}{2}}f_{x_i} [f(x) - f(0)]_{x_i} \xi^2 dx.
\]
Therefore, the integration by parts yields
\[
\begin{split}
&\int_{B_\rho(x_0)} v^{\frac{p+2}{2}} \xi^2 dx = - \int_{B_\rho(x_0)} [f(x) - f(x_0)] \Big[v^{\frac{p}{2}}\Delta f \, \xi^2 
+ pv^{\frac{p-2}{2}}f_{x_i}f_{x_l}f_{x_lx_i} \xi^2 + 2v^{\frac{p}{2}}f_{x_i} \xi\xi_{x_i} \Big]dx\\
& \leq \text{osc}_{B_\rho(x_0)} f  \int_{B_\rho(x_0)} \Big[(\sqrt{n}+p )v^{\frac{p}{2}} |\nabla^2 f|  \xi^2 + 2v^{\frac{p+1}{2}}|\nabla \xi| \xi  \Big] dx\\
& \leq \frac{1}{2}
\int_{B_\rho(x_0)} v^{\frac{p+2}{2}} \xi^2dx 
+ (\sqrt{n} +p)^2\big(\text{osc}_{B_R} f\big)^2 \int_{B_R}\Big[v^{\frac{p-2}{2}} |\nabla^2 f|^2 \xi^2 + v^{\frac{p}{2}} |\nabla \xi|^2 \Big] dx.
\end{split}
\]
The lemma then follows. 
\end{proof}

The next interpolation result is extracted from \cite[page 55]{BDM}.

\begin{lemma}\label{interpolation}
Let  $f\in L^\infty(B_{R})$ with $R>0$. Assume that there exist constants $q>p>0$ and $\gamma>0$ such that
\begin{equation}\label{L^q-average}
 \|f\|_{L^\infty(B_{(1-\sigma)r})}\leq \frac{\gamma}{(\sigma r)^{\frac{n}{q}}} \Big( \int_{B_r} |f|^q \, dx\Big)^{\frac{1}{q}}
\end{equation}
for every $r\in (0, R)$ and every $\sigma \in (0,1)$. Then we have
\begin{equation*}
 \|f\|_{L^\infty(B_{\frac{R}{2}})}\leq \frac{\gamma'}{R^{\frac{n}{p}}} \Big( \int_{B_{R}} |f|^p \, dx\Big)^{\frac{1}{p}},
\end{equation*}
where $\gamma' =  \frac{p }{q-p} 2^{\frac{n}{p}} \Big(  2^{\frac{n}{p} +1} \frac{q-p}{q} \gamma \Big)^{\frac{q}{p}}$. In particular, $\gamma' = 8^{\frac{n}{p}} \gamma^2$ if $q=2p$.
\end{lemma}
\begin{proof}
The proof of this lemma  for particular $q=p+2$ is in \cite[page 55]{BDM}.
For the sake of completeness, we include the same arguments for all $q>p$ here.

Let $G = ( \int_{B_{R}} |f|^p \, dx)^{\frac{1}{p}}$, and for  $s=0, 1,\dots $, 
 \[
  r_s = \frac{R}{2} \sum_{i=0}^{s} 2^{-i}, \quad   F_s = \|f\|_{L^\infty(B_{r_s})}.
 \]
Then by applying \eqref{L^q-average} to $r = r_{s+1}$ and $\sigma r = r_{s+1} - r_s = R/ 2^{s+2}$, we obtain 
that
\[
 F_s \leq  2^{\frac{n s}{q}} \Big(\frac{4}{R}\Big)^{\frac{n}{q}} \gamma \Big( \int_{B_{r_{s+1}}} |f|^q \, dx\Big)^{\frac{1}{q}}  \leq 2^{\frac{n s}{q}} \Big(\frac{4}{R}\Big)^{\frac{n}{q}} \gamma F_{s+1}^{\frac{q-p}{q}} G^{\frac{p}{q}}.
\]
Using Young's inequality, it follows for any $\delta>0$ that
\begin{equation}\label{iteration}
 F_s \leq \delta F_{s+1} +  2^{\frac{n s}{p}} \Theta   G\quad\mbox{for}\quad s=0,1,...
\end{equation}
with
$
\Theta=\frac{p}{q}
  \Big(\frac{q-p}{\delta q}\Big)^{\frac{q-p}{q}} \Big(\frac{4}{R}\Big)^{\frac{n}{p}} \gamma^{\frac{q}{p}}$.
Thus by iterating the relation \eqref{iteration}, we get
\begin{align*}
 F_0 \leq \delta^s F_s + \Theta G \sum_{i=0}^{s-1} \big(\delta  2^{\frac{n}{p}})^i 
 \leq \delta^s \|f\|_{L^\infty(B_R)} + \Theta G \sum_{i=0}^{s-1} \big(\delta  2^{\frac{n}{p}})^i
\end{align*}
for any $s=1,2,...$ Then by choosing $\delta = 2^{-(\frac{n}{p} +1)}$ and letting $s\to \infty$, we deduce that
\[
 F_0\leq 2 \Theta G =  \frac{2 p}{ q } \Big( 2^{\frac{n}{p} +1} \frac{q-p}{q} \Big)^{\frac{q-p}{q}} \big(\frac{4}{R}\big)^{\frac{n}{p}} \gamma^{\frac{q}{p}}G
 =\frac{\gamma'}{R^{\frac{n}{p}}} G.
\]
This completes the proof as $F_0 = \|f\|_{L^\infty(B_{\frac{R}{2}})}$.
\end{proof}

\section{Proofs of main theorems}\label{sec:proof-main-theorem}

We start with proving Theorem~ \ref{Lip}. Our proof consists of two main steps, and the crucial one is given in the following proposition.

\begin{proposition} \label{prop:Lip} Assume that \textup{(H2)--(H5)} hold.
Let $u$ be a  weak solution of \eqref{ref-eqn} that satisfies 
\beq\label{inter-rough-ass}
\int_{B_\frac{5}{2}} |\nabla u|^{2(p+\bar q )}\, dx\leq \bar M\quad \mbox{for some } \bar q >\max{\{1, \frac{n}{2}\}}.
\eeq
Then there exists $C>0$ depending only on $n$, $p$,  $\bar q$, $\gamma_0$, $\gamma_1$,  and $\bar M$ such that
inequality \eqref{Win} holds true.
\end{proposition}
\begin{proof}
The proof uses Lemma~\ref{DiGiorgi-Class} and De Giorgi's iteration. We provide full calculations here.
Without loss of generality, we assume $x_0=0$.

  Let $v= w^{p/2} = |\nabla u|^p$. For each $k> 0$ and $r >0$, denote
\[ 
 A_{k,r} = \{x \in B_r: v(x) > k\}.
\]

Let $K$ be a positive number which will be determined. 
Let  $\zeta(s)$ be a smooth cut-off function on $\R$ which equals  unity for $s \leq 0$, vanishes for $s\ge \frac{1}{2}$, and 
$|\zeta'| \leq c$ for some constant $c>0$.

Let us fix $R\in (0,3/2]$ and $\sigma\in (0,1)$. Then for $i = 0, 1, 2,\dots$, we denote
\[
\rho_i =\Big (1 -\sigma + \frac{\sigma}{2^{i}}\Big)R, \quad \bar{\rho}_{i} = \frac{\rho_i + \rho_{i+1}}{2},\quad  
\xi_i (y)= \zeta\left(\frac{2^{i+1}}{\sigma R}(|y|-\rho_{i+1})   \right),
\]
\[
k_i = K\Big(1 - \frac{1}{2^i}\Big), \quad
v_i  = (v -k_i)^+ .\]
Then $ \rho_{i+1} < \bar{\rho}_i < \rho_i$, the function $\xi_i$ vanishes outside $B_{\bar{\rho}_i}$, equals unity on $B_{\rho_{i+1}}$, and
\begin{equation} \label{test-grad.est}
0\le \xi_i\le 1,\quad |\nabla \xi_i| \leq \frac{c\,2^{i+1}}{\sigma R}\quad \text{on }  B_3.
\end{equation}
Let $n/2 <q \leq \infty$. By applying Lemma~\ref{DiGiorgi-Class} with $k=k_{i+1}>0$, $\xi=\xi_i$ and by using \eqref{test-grad.est} together with H\"older's inequality,  we obtain 
\begin{align}\label{useM_0}
 \int_{B_3} &|\nabla (v_{i+1} \xi_i)|^2  dx 
 \leq C \int_{B_3} v_{i+1}^2|\nabla\xi_i|^2 dx
+C \int_{B_3} (w^p + w^{p+1})\chi_{v >k_{i+1}}(x)\xi_i^2 dx\\
& \leq C \Big[\frac{4^i}{(\sigma R)^2}\int_{A_{k_{i+1}, \rho_i }} v_{i+1}^2 dx
+ \int_{A_{k_{i+1}, \rho_i}} w^p  dx +\int_{A_{k_{i+1}, \rho_i}}  w^{\frac{p}{q} +1} w^{\frac{p}{q'}} dx\Big]\nonumber\\
& \leq C \left\{\Big(\int_{A_{k_{i+1}, \rho_i }} |\nabla u|^{2p} dx\Big)^{\frac{1}{q}}
\Big[ \frac{4^i}{(\sigma R)^2} \Big(\int_{A_{k_{i+1}, \rho_i }} v_{i+1}^2 dx\Big)^{\frac{1}{q'}}
+ \Big(\int_{A_{k_{i+1}, \rho_i}} w^p  dx\Big)^{\frac{1}{q'}}\Big]\right.\nonumber\\
&\quad \left. +\Big(\int_{A_{k_{i+1}, \rho_i}} |\nabla u|^{2(p+q) } dx\Big)^{\frac{1}{q}} 
\Big(\int_{A_{k_{i+1}, \rho_i}} w^p  dx\Big)^{\frac{1}{q'}} \right\}.\nonumber
\end{align}
Let  $M_q(R) = \||\nabla u|^{\frac{2p}{q}}\|_{L^q(B_R)}   + \||\nabla u|^{\frac{2p}{q} +2} \|_{L^q(B_R)}$ with the convention that $M_\infty(R) = 1 + \|\nabla u\|_{L^\infty(B_R)}^2$. Then it follows from \eqref{useM_0} that
\begin{align}\label{ineq.Di}
\int_{B_3} &|\nabla (v_{i+1} \xi_i)|^2  dx 
   \leq C M_q(R)  \left\{ \frac{4^i}{(\sigma R)^2} \Big(\int_{A_{k_{i+1}, \rho_i }} v_{i+1}^2 dx\Big)^{\frac{1}{q'}}
+ \Big(\int_{A_{k_{i+1}, \rho_i}} v^2  dx\Big)^{\frac{1}{q'}}
 \right\}\\
 & \leq C M_q(R) \left\{ \frac{4^i}{(\sigma R)^2} \Big(\int_{A_{k_{i+1}, \rho_i }} v_{i+1}^2 dx\Big)^{\frac{1}{q'}}
+ \Big(\int_{A_{k_{i+1}, \rho_i}} v_{i+1}^2  dx + K^2 |A_{k_{i+1}, \rho_i}|\Big)^{\frac{1}{q'}}
 \right\}\nonumber\\
 & \leq C M_q(R) \left\{ \frac{4^i}{(\sigma R)^2} \Big(\int_{A_{k_{i+1}, \rho_i }} v_{i+1}^2 dx\Big)^{\frac{1}{q'}}
+ \Big( K^2 |A_{k_{i+1}, \rho_i}|\Big)^{\frac{1}{q'}}
 \right\},\nonumber
\end{align}
where  $C$ depends only on $n$, $p$, $\gamma_0$, $\gamma_1$.

We next show that \eqref{ineq.Di} implies the desired estimate \eqref{Win}. For this, let us define
\[
  J_i = \int_{A_{k_i}, \rho_i}v_i^2 \,dx.
\]
By properties of $\xi_i$,  Sobolev's embedding  $W^{1,\frac{2n}{n+2}}(B_3)\hookrightarrow L^2(B_3)$ when $n\geq 2$, and H\"older's inequality, we have
\begin{equation*} 
\begin{split}
J_{i+1} & \leq \int_{B_3}  (v _{i+1}\xi_i)^2  dx\leq C \Big(\int_{B_3} |\nabla  (v _{i+1}\xi_i)|^{\frac{2n}{n+2}}dx\Big)^{\frac{n+2}{n}}
\leq   C |A_{k_{i+1}, \rho_{i}}|^{\frac2n} \int_{B_3} |\nabla  (v _{i+1}\xi_i)|^2dx.
\end{split}
\end{equation*}
We note that this estimate for $J_{i+1}$ still holds true when $n=1$. Indeed, in that case we can use the Sobolev's embedding  $W^{n,1}(B_3)\hookrightarrow C(\overline{B_3})$ and H\"older's inequality
to obtain
\begin{equation*} 
\begin{split}
J_{i+1} & \leq \int_{B_3}  (v _{i+1}\xi_i)^2  dx\leq  |A_{k_{i+1}, \rho_{i}}|\, \|v _{i+1}\xi_i\|_{L^\infty(B_3)}^2
\leq C |A_{k_{i+1}, \rho_{i}}|  \Big(\int_{B_3} |\nabla  (v _{i+1}\xi_i)| dx\Big)^2\\
&\leq   C |A_{k_{i+1}, \rho_{i}}|^{2} \int_{B_3} |\nabla  (v _{i+1}\xi_i)|^2dx.
\end{split}
\end{equation*}
It follows from the estimate for $J_{i+1}$,  \eqref{ineq.Di} and the fact $\int_{A_{k_{i+1}, \rho_i }} v_{i+1}^2 dx \leq 
\int_{A_{k_{i+1}, \rho_i }} v_i^2 dx \leq J_i$ that
\begin{equation} \label{A.est}
\begin{split}
J_{i+1} &\leq  C M_q(R) |A_{k_{i+1}, \rho_{i}}|^{\frac2n} 
 \Big [ \frac{4^i}{(\sigma R)^2} J_i^{\frac{1}{q'}}
+ \big( K^2 |A_{k_{i+1}, \rho_i}|\big)^{\frac{1}{q'}} \Big].
\end{split}
\end{equation}
The monotonicity of $k_i$ implies that
\begin{equation*} 
J_i \geq \int_{A_{k_{i+1}, \rho_i}}(v - k_{i})^2 dx
\geq (k_{i+1} - k_i)^2|A_{k_{i+1}, \rho_i}| = 4^{-(i+1)}K^2|A_{k_{i+1}, \rho_i}|,
\end{equation*}
which gives
\beq\label{below.A}
|A_{k_{i+1}, \rho_i}|\le 4^{i+1}K^{-2} J_i.
\eeq

From  \eqref{A.est} and \eqref{below.A}, we deduce that 
\begin{equation} \label{iter-for} 
J_{i+1} \leq  C M_q(R) |A_{k_{i+1}, \rho_{i}}|^{\frac2n} 
 \frac{4^i}{(\sigma R)^2} J_i^{\frac{1}{q'}} \leq  C M_q(R) (\sigma R)^{-2} K^{\frac{-4}{n}} B^i J_{i}^{1+\kappa}, 
\end{equation}
where   $B = 4^{\frac{2}{n} +1}$ and $\kappa = \frac{2}{n}-{\frac{1}{q}}$. Note that as $q>n/2$, we have $\kappa>0$.

By iterating formula \eqref{iter-for}, we see that
\[
J_i \leq \big[ C M_q(R)(\sigma R)^{-2} K^{\frac{-4}{n}}  \big]^{\frac{(1+\kappa)^i -1}{\kappa}}B^{\frac{(1+\kappa)^i -1}{\kappa^2} -
\frac{i}{\kappa}} J_0^{(1+\kappa)^i} \quad \mbox{for all } i =0, 1,\dots
\] 
Next, select
\[
K = \big[ C M_q(R) (\sigma R)^{-2} \big]^{\frac{n}{4}}  B^{\frac{n}{4\kappa}}\left [ \int_{B_R} |\nabla u|^{2p}dx \right]^{\frac{n\kappa}{4}}
\]
which ensures
\begin{equation*}
J_0 =\int_{A_{k_0,\rho_0}} |\nabla u|^{2p} dx \leq \int_{B_R} |\nabla u|^{2p} dx  
=  \big[ C M_q(R) (\sigma R)^{-2} K^{\frac{-4}{n}}  \big ]^{-\frac{1}{\kappa}} B^{-\frac{1}{\kappa^2}}=:\Lambda.
\end{equation*}
Therefore, we obtain
\[
J_{i} \leq \Lambda B^{-\frac{i}{\kappa}} \rightarrow 0, \quad \text{as} \quad 
i \rightarrow \infty.
\]
Hence, we conclude that
\[
 |\nabla u (x)|^{p} =v(x)\leq K\quad  \text{a.e. in } B_{(1-\sigma)R}.
 \]
Thus we have proved that
\beq\label{initial-bound}
|\nabla u(x)| \leq   \frac{1}{(\sigma R)^{\frac{n}{2p}}}
\left[ C B^{\frac{1}{\kappa}} M_q(R) \Big(\int_{B_R} |\nabla u|^{2p} dx\Big)^\kappa \right]^{\frac{n}{4p}}\quad 
\text{a.e. in } B_{(1-\sigma)R}
\eeq
for every $R\in (0, 3/2]$, $\sigma\in (0,1)$ and $2/n<q\leq \infty$.
 By taking  $R=3/2$, $\sigma= 1/3$, $q=\bar q$ and using assumption \eqref{inter-rough-ass}, we see that
the right hand side of  \eqref{initial-bound} is bounded. As a consequence,
  there exists a constant 
 $C_*$ depending only on $n$, $p$, $\bar q$, $\gamma_0$, $\gamma_1$ and $\bar M$ such that
 \begin{equation}\label{eq:infinity-est}
 \|\nabla u\|_{L^\infty(B_1)}\leq C_*.
 \end{equation}
Next,  we infer from  \eqref{initial-bound} with $q=\infty$, the fact $\kappa =\frac2n -\frac1q$ and \eqref{eq:infinity-est} that
\begin{align*}
\|\nabla u\|_{ L^\infty(B_{(1-\sigma)R})}
\leq \frac{\big[ C B^{\frac{n}{2}} M_\infty(1)\big]^{\frac{n}{4p}}}{(\sigma R)^{\frac{n}{2p}}}  
\Big(\int_{B_R} |\nabla u|^{2p} dx\Big)^{\frac{1}{2p}}\leq  \frac{\gamma}{(\sigma R)^{\frac{n}{2p}}}
\Big(\int_{B_R} |\nabla u|^{2p} dx\Big)^{\frac{1}{2p}}
\end{align*}
for every $R\in (0,1]$ and 
every $\sigma\in (0,1)$, where $\gamma =\big[ C B^{\frac{n}{2}} (1+ C_*^2)\big]^{\frac{n}{4p}}$.
Hence,  we can use
 the interpolation result in Lemma~\ref{interpolation}  for $q=2p$  to get 
\begin{equation*}
 \|\nabla u\|_{L^\infty(B_{\frac{R}{2}})}\leq \frac{8^{\frac{n}{p}} \gamma^2}{R^{\frac{n}{p}}} \Big( \int_{B_{R}} |\nabla u|^p \, dx
 \Big)^{\frac{1}{p}}\quad \mbox{for all } 0<R\leq 1.
\end{equation*}
Therefore, the proof is complete.
\end{proof}

We are now ready to prove our main results.

\begin{proof} [\textbf{Proof of Theorem \ref{Lip}}] 

Thanks to Proposition~\ref{prop:Lip}, it remains to verify condition \eqref{inter-rough-ass}. We  complete this step by claiming that  
for any positive integer $m$  there exists a constant $M>0$ depending only on $n$, $p$, $m$,  $\gamma_0$, $\gamma_1$ and $M_0$ such that
\begin{equation}\label{L^q-est}
\int_{B_{\frac52}} |\nabla u|^{p+2m} \, dx\leq M. 
\end{equation}
Indeed, let us fix  $m\in \{1,2,\dots\}$. As in the proof of Lemma~\ref{test-lemma}, we can assume 
that $u\in C^2(B_3)$ with $|\nabla u|>0$. Let $x_0\in B_{\frac52}$ and $0<\rho\leq 1/8$ be arbitrary,  which ensure 
that $B_{2\rho}(x_0) \subset B_{\frac{11}{4}}$.  Consider 
  $s\geq 0$ and   a nonnegative function  $\xi \in C^\infty_0(B_\rho(x_0))$. 
Then by applying Lemma \ref{test-lemma}
with $\beta(w) = (w +\delta)^s$ and letting $\delta \rightarrow 0^+$, we obtain
\begin{align*} 
I &\eqdef \int_{B_3} w^{\frac{p-2 +2s}{2}} |\nabla^2 u|^2  \xi^2 dx  
+s \int_{B_3} w^{\frac{p-4+2 s}{2}}|\nabla w|^2\xi^2 dx\\
&\leq  C(s+1) \int_{B_3}\big(w^{\frac{p+2s}{2}} + w^{\frac{p+2+2s}{2}} \big)\xi^2 dx\nonumber\\
&+C \int_{B_3} w^{\frac{p-1+2s}{2}}|\nabla^2 u||\nabla \xi| \xi dx + \int_{B_3} \big( w^{\frac{p+2s}{2}} +w^{\frac{p+1+2s}{2}} 
\big)|\nabla \xi| \xi dx
\nonumber
\end{align*}
where we have used $|\nabla w| \leq C w^{\frac12} |\nabla^2 u|$. It follows from Young's inequality and by moving some terms around that
\begin{align} \label{Lq-1}
I  &\leq C(s+1) \int_{B_3}\big(w^{\frac{p+2s}{2}} + w^{\frac{p+2+2s}{2}} \big)\xi^2 dx + 
C\int_{B_3} w^{\frac{p+2s}{2}} |\nabla \xi|^2 dx
\end{align}
with $C$  depending only on $n$, $\gamma_0$ and $\gamma_1$. 
Next, applying Lemma~\ref{embed-iq} for $f=u$ and  with test function $w^{s/2}\xi$, we get 
\begin{equation} \label{Lq-2}
\int_{B_3} w^{\frac{p+2 +2s}{2}}\xi^2 dx \leq 
4( \sqrt{n} +p)^2 \big(\text{osc}_{B_\rho(x_0)} u \big)^2\Big [ (s+1) I + \int_{B_3}w^{\frac{p+2s}{2}} |\nabla \xi|^2 dx \Big].
\end{equation}
Owing  to assumption \eqref{rough-ass} and the  fact $B_\rho(x_0)\subset B_{\frac{21}{8}}$, we can infer from Theorem~\ref{thm:Holder} that 
$\text{osc}_{B_\rho(x_0)} u\leq C_0 \rho^\alpha$. Thus we deduce  from  \eqref{Lq-2} and  \eqref{Lq-1} that
\begin{equation} \label{Lq-3}
\int_{B_3} w^{\frac{p+2 + 2s}{2}} \xi^2 dx \leq \gamma\, \rho^{2\alpha}(s+1)^2 \Big[ \int_{B_3} w^{\frac{p+2+2s}{2}} \xi^2 dx + \int_{B_3} w^{\frac{p+2s}{2}} \big(\xi^2 +|\nabla \xi|^2 \big) dx \Big],
\end{equation}
where $\gamma$ and $\alpha$ depend only on $n$, $p$, $\gamma_0$, $\gamma_1$ and $M_0$.
Now let $R_0 =  \min\{(2\gamma)^{-\frac{1}{2\alpha}}, \frac18\}$, and 
\[
R_s =  R_0 (1+s)^{-\frac{1}{\alpha}}\quad \mbox{for } s\geq 0.
\]
Let $\xi_s\in C_0^\infty(B_{R_s}(x_0))$ be the standard cut-off function which equals one in $B_{R_{s+1}}(x_0)$,  and
\[
|\nabla \xi_s| \leq \frac{2}{R_s -R_{s+1}}.
\] 
Then by using this test function    in \eqref{Lq-3}, we obtain
\[
\int_{B_{R_s}(x_0)} w^{\frac{p+2 + 2s}{2}} \xi_s^2 dx \leq \frac{1}{2} \int_{B_{R_s(x_0)}} w^{\frac{p+2+2s}{2}} \xi_s^2 dx + \frac{1}{2} 
\int_{B_{R_s(x_0)}} w^{\frac{p+2s}{2}} \big(\xi_s^2 +|\nabla \xi_s|^2 \big) dx
\]
yielding
\[
\int_{B_{R_{s+1}(x_0)}} |\nabla u|^{p + 2s+2}  dx 
 \leq \frac{5}{(R_s - R_{s+1})^2} \int_{B_{R_s}(x_0)} |\nabla u|^{p+2s}dx\quad \forall s\geq 0.
\]
 By iterating this inequality from $s=0$  to $s=m-1$ and using the fact $R_{i-1}- R_i \geq (1+i)^{-\frac{2}{\alpha}}$, we conclude that
\begin{equation*}
\int_{B_{R_m(x_0)}} |\nabla u|^{p+2m} dx \leq \frac{5^m}{\Pi_{i=1}^m (R_{i-1} -R_{i})^2}
\int_{B_{R_0}(x_0)} |\nabla u|^{p}  dx\leq 5^m [(m+1)!]^{\frac{2}{\alpha}} \int_{B_{R_0}(x_0)} |\nabla u|^{p}  dx.
\end{equation*}
As $B_{2 R_0}(x_0)\subset B_{\frac{11}{4}}$, we can use  Lemma~\ref{Lp-est} together with  assumption \eqref{rough-ass}   
 to bound the above right-hand side. Consequently, we obtain 
\begin{equation}\label{L^m-R_m}
\int_{B_{R_m(x_0)}} |\nabla u|^{p+2m} dx \leq C(n, p, m, \gamma_0, \gamma_1, M_0)\quad \mbox{for all } x_0\in B_{\frac52}.
\end{equation}
Now by covering $B_{\frac52}$ with a finite number of balls $B_{R_m}(x_i)$ with $x_i\in B_{\frac52}$, we deduce 
claim \eqref{L^q-est} from  \eqref{L^m-R_m}. The proof is therefore complete.
\end{proof}

\begin{proof} [\textbf{Proof of Theorem \ref{Lip-plus}}] 
The proof  is a direct consequence of that of Proposition~\ref{prop:Lip}. Observe that 
in the proof of Proposition~\ref{prop:Lip},  assumption \eqref{inter-rough-ass} is only used to control the term 
$w^{p+1}$ in \eqref{useM_0} which comes from  Lemma~\ref{DiGiorgi-Class}.
Thus by  using 
\eqref{strengthened-est} in place of Lemma~\ref{DiGiorgi-Class}, 
we see that \eqref{ineq.Di} holds for  $q=\infty$ and with $M_\infty(R)$ being replaced by $1$. Therefore, estimate \eqref{iter-for} is valid without the term $M_q(R)$ and for  $\kappa = 2/n$.
With this change and  by repeating the  arguments after \eqref{iter-for}, we obtain \eqref{Win}. Note also that assumption  (H1) is 
not needed  since Lemma~\ref{Lp-est}, Theorem~\ref{thm:Holder} and Lemma~\ref{embed-iq} are not used in the proof.
\end{proof}


\textbf{Acknowledgement.} L.H. gratefully acknowledges the support provided by  NSF grant DMS-1412796. T.N. gratefully acknowledges 
the support by a grant  from the Simons Foundation (\# 318995).  This work is completed while the second author is visiting Tan Tao university and he would like to thank the
institution for the kind hospitality.


\begin{thebibliography}{10}
\bibitem{BDM}
  T.~Bhattacharya, E.~DiBenedetto and J.~ Manfredi.  {\it Limits as $p\to \infty$ of $\Delta_p u =f$ and related extremal
 problems,}  Rend. Sem. Mat. Univ. Politec. Torino  1989,  Special Issue, 15--68 (1991).
 
 

\bibitem{CP} L.A.~Caffarelli and I.~Peral. 
{\it    On $W^{1,p}$ estimates for elliptic equations in divergence  form,}
 Comm. Pure Appl. Math.  51  (1998),  no. 1, 1--21.


\bibitem{D} E.~DiBenedetto. {\it  $C^{1+\alpha}$ local regularity of weak solutions of degenerate
 elliptic equations,} Nonlinear Anal.  7  (1983),  no. 8, 827--850.

\bibitem{E1}  
  L.~ Evans.  {\it  A new proof of local $C^{1,\alpha}$ regularity for solutions of
 certain degenerate elliptic p.d.e.,}
 J. Differential Equations  45  (1982),  no. 3, 356--373.
 
\bibitem{HNP1} L. Hoang, T. Nguyen and T. Phan. {\it Gradient estimates and global existence of smooth solutions to a cross-diffusion system,} To appear in SIAM J. Math. Anal.

\bibitem{HNP2} L. Hoang, T. Nguyen and T. Phan. {\it Interior gradient estimates for quasilinear elliptic equations in divergence form,} Preprint, 2015.

\bibitem{La} O.~ Ladyzhenskaya and N.~ Uraltceva.
{\it Linear and quasilinear elliptic equations,} Academic Press, New York, 1968.

 \bibitem{Le} J.~Lewis. {\it Regularity of the derivatives of solutions to certain degenerate
 elliptic equations,}
 Indiana Univ. Math. J.  32  (1983),  no. 6, 849--858.
 
 
 \bibitem{M} J.~Manfredi.  {\it Regularity for minima of functionals with $p$-growth,}
 J. Differential Equations  76  (1988),  no. 2, 203--212.


\bibitem{Serrin} J.~Serrin. {\it Local behavior of solutions of quasi-linear elliptic equations,} Acta Math. 111 (1964), 247--302.

\bibitem{Tol1} P.~ Tolksdorf. {\it Everywhere regularity for some quasilinear systems with  a lack of ellipticity,} Ann. Mat. Pura Appl., 134, (1983), 241--266.

\bibitem{Tol2} P.~ Tolksdorf. {\it Regularity for a more general class of quasilinear 
elliptic equations,} J. Diff. Eqn., 51, (1984), 126--150.

\bibitem{Uh} K. Uhlenbeck. {\it Regularity for a class of non-linear elliptic systems,}
Acta Math. 138 (1977), 219--240.

\bibitem{Ur} N.~ Uraltceva.
{\it Degenerate quasilinear elliptic systems,}  Zap. Nauchn. Sem. Leningrad. Otdel. Mat. Inst. Steklov 7 (1968), 184--222. [In Russian]


\end{thebibliography}
\end{document}